\documentclass{amsart}
\usepackage{amsmath}
\usepackage{amsfonts}
\usepackage{amssymb}
  \usepackage{paralist}
  \usepackage{graphics} 
  \usepackage{epsfig} 
\usepackage{graphicx}  \usepackage{epstopdf}
 \usepackage[colorlinks=true]{hyperref}
   
   \usepackage{tikz} 
\hypersetup{urlcolor=blue, citecolor=red}

\newtheorem{theorem}{Theorem}[section]
\newtheorem{corollary}{Corollary}

\newtheorem{proposition}{Proposition}

\theoremstyle{definition}

\newtheorem{remark}{Remark}

\newtheorem{example}{Example}

\newcommand{\F}{\mathbb{F}}

\newcommand{\bpf}{\begin{proof}}
\newcommand{\epf}{\end{proof}}

\newcommand{\rmv}[1]{}

\title[LRCs from fiber products of curves] 
      {Locally recoverable codes with availability $t\geq 2$ from fiber products of curves}

\author[Kathryn Haymaker, Beth Malmskog, and Gretchen Matthews]{}


 \email{kathryn.haymaker@villanova.edu}
 \email{beth.malmskog@gmail.com}
 \email{gmatthe@clemson.edu}

\thanks{The second author is supported by NSA grant H98230-16-1-0300}

\thanks{$^*$ Corresponding author: Beth Malmskog}

\begin{document}
\maketitle

\centerline{\scshape Kathryn Haymaker}
\medskip
{\footnotesize
 \centerline{Department of Mathematics and Statistics}
   \centerline{Villanova University}
   \centerline{Villanova, PA  19085, USA}
   }
\medskip

\centerline{\scshape Beth Malmskog$^*$}
\medskip
{\footnotesize
 \centerline{Department of Mathematics and Computer Science}
   \centerline{Colorado College}
   \centerline{Colorado Springs, CO 80903 USA}
} 

\medskip

\centerline{\scshape Gretchen L. Matthews}
\medskip
{\footnotesize
 \centerline{ Department of Mathematical Sciences}
   \centerline{Clemson University }
   \centerline{Clemson, SC 29634, USA}
}

\bigskip

\begin{abstract}
We generalize the construction of locally recoverable codes on algebraic curves given by Barg, Tamo and Vl\u{a}du\c{t} \cite{Barg16} to those with arbitrarily many recovery sets by exploiting the structure of fiber products of curves.  Employing maximal curves, we create several new families of locally recoverable codes with multiple recovery sets, including codes with two recovery sets from the generalized Giulietti and Korchm\'{a}ros (GK) curves and the Suzuki curves, and new locally recoverable codes with many recovery sets based on the Hermitian curve, using a fiber product construction of van der Geer and van der Vlugt. In addition, we consider the relationship between local error recovery and global error correction as well as the availability required to locally recover any pattern of a fixed number of erasures. 
\end{abstract}

%

\section{Introduction} \label{Introduction}

Codes for local recovery were introduced in the context of distributed storage systems, where there is a need to repair 
a single erasure or small number of erasures by accessing a few coordinates of the received word, rather than accessing the entire received word. A linear code is said to have locality $r$ if for each coordinate $i$ of a codeword, there is a set of $r$ helper coordinates, called a recovery or helper set, so that in any codeword, the symbol in position $i$ can be recovered from the symbols in the helper set.

If an element of a recovery set becomes unavailable, local recovery may not be possible. This leads to what is known as the availability problem and the need for multiple recovery sets. A code is said to have availability $t$ if each coordinate has $t$ disjoint recovery sets.  In \cite{Barg16}, the authors construct locally recoverable codes with availability $t=2$ based on fiber products of curves and propose a group-theoretic perspective on the construction, whereby a curve can  sometimes be expressed as a fiber product of its quotient curves by certain subgroups of the automorphism group of the curve; see \cite{Barg16extended} for an extended version. In particular, given a curve $\mathcal{X}$ with automorphism group $Aut(\mathcal{X})$ containing a semi-direct product of two subgroups, a locally recoverable code with availability $2$ can be formed by considering the fixed fields of the function field of $\mathcal{X}$ of the subgroups in the semi-direct product. It is remarked that both perspectives can be extended in a straightforward way to provide multiple recovery sets, meaning $t>2$. 

In Section \ref{construction_subsection}, we carry out the generalization of the fiber product approach for $t\geq 2$. This also generalizes the group-theoretic perspective; however, we note that when $t>2$, associativity of a particular semi-direct product is required. This issue does not arise in the case where $t=2$, and, hence, is not addressed explicitly in \cite{Barg16}. See Remark  \ref{group bad remark} for more discussion. 

A key difference in our work and that of \cite{Barg16} is that we employ a different method to bound the code parameters than in \cite{Barg16}.  This results in new, potentially sharper bounds on minimum distance.  The bounds in \cite{Barg16} require knowledge of the degree of a function $x:\mathcal{X}\rightarrow \mathbb{P}^1$ where $x$ generates a particular field extension. Practically, this function can be quite difficult to construct. Our bounds avoid the issue and do not require that this function is known.  A more thorough discussion of this issue appears in Remark \ref{differences remark}.

We also consider properties of multiple local recovery sets such as  the availability required to locally recover any pattern of a fixed number of erasures; this may be found in Section \ref{recovery_issues}. There, the relationship between local error recovery and global error correction is investigated. 

Sections \ref{GK_section}, \ref{Suzuki_section}, and \ref{LRCt_section} demonstrate the `top-down' (group-theoretic) and `bottom-up' (fiber product-centered) approaches discussed in Section \ref{construction_subsection}. In Section \ref{GK_section}, we employ generalized GK curves in the fiber product construction. While the construction from GK curves is quite explicit, the recovery sets are unbalanced, meaning there is a large difference between the cardinalities of the two recovery sets associated with a particular coordinate; this may adversely impact other code parameters and may be less than optimal for applications. We illustrate the automorphism group perspective in Section \ref{Suzuki_section}, constructing theoretical codes on the Suzuki curve from its automorphism group that provide balanced recovery sets.  However, this section highlights the difficulty of using the automorphism group construction to explicitly construct LRCs: even with knowledge of the quotient curves, constructing the necessary maps can be quite difficult.  Finally, in Section \ref{LRCt_section}, we obtain locally recoverable codes on the Hermitian curve $\mathcal{H}_{p^{t}}$ with availability $t$ for arbitrarily large $t \geq 2$, which are explicit in nature and balanced, using the full power of the LRC($t$) construction. This differs from the recent work \cite{IPAM} which uses elliptic curves to construct LRC($2$)s as well as surfaces and other curves to construct LRC($1$)s. 

The curves considered in this paper are maximal, meaning they have as many points as possible over a given finite field.  Maximal curves have played an important role in the construction of algebraic geometry codes, as they support the construction of long codes over relatively small fields, with minimum distance bounded below based on geometric arguments. Their utility in the construction of locally recoverable codes comes from this as well.  Though the constructions in this paper do not require the curves involved to be maximal, all examples in the paper are based on maximal curves.

Our findings are summarized in Section \ref{Conclusion}.

\section{Notation and background}\label{Notation}

 We use curves over finite fields and maps of curves to construct locally recoverable codes with multiple recovery sets.  The following notation and background will be used throughout.  See \cite{Stichtenoth} and \cite[Chapter 3]{Liu} for more background and proofs.

Let $p$ be a prime number, and let $q=p^s$ for some $s\in\mathbb{N}$. Let $\mathbb{F}_q$ be the field with $q$ elements. For any natural number $n$, let $[n]=\{ 1, \dots, n\}$.  A linear code $C$ of length $n$ over $\F_q$ is a locally recoverable code, or LRC, with locality $r$ if and only if for all $\mathbf{c}=(c_1, c_2,\dots, c_n) \in C$ and for all $i \in [n]$ there exists a set 
\[ A_i \subseteq [n] \setminus \{ i \}, \]  
$|A_i|=r$, such that $c_i=\phi_i(\mathbf{c}|_{A_i})$ for some function $\phi_i: \F_q^{r} \rightarrow \F_q$. The idea is that the codeword symbol $c_i$ can be recovered from the symbols indexed by elements of $A_i$ without access to the other coordinates of the received word. The helper set $A_i$ is called a recovery set for the $i$th position. 

 We say that a locally recoverable code $C \subseteq \F_q^n$  has availability $t$ with locality $(r_1, \dots, r_t)$ provided 
 for all $i \in [n]$ there exists $A_{i1}, \dots, A_{it} \subseteq [n] \setminus \{i \}$ with $|A_{ij}|=r_{j}$, $A_{ij} \cap A_{ik} = \emptyset$ for $j \neq k$,  and for all $\mathbf{c} \in C$,  $c_i=\phi_{ij}(\mathbf{c}|_{A_{ij}})$ for some function $\phi_{ij}: \F_q^{r_j} \rightarrow \F_q$. Such a code is called an LRC($t$), or a locally recoverable code with availability $t$, to emphasize that each coordinate has $t$ disjoint recovery sets. Let $B_{ij} = A_{ij} \cup \{i\}$ for each $i\in [n]$, $j\in [t]$. 
 
Let $\mathcal{X}$ be a smooth, projective, absolutely irreducible algebraic curve defined over $\F_q$, of genus $g_{\mathcal{X}}$.  Let $\F_q(\mathcal{X})$ denote the field of functions on $\mathcal{X}$ defined over $\F_q$.  For any $i\in\mathbb{N}$, let $\mathcal{X}(\F_{q^i})$ denote the set of $\F_{q^i}$-rational points of $\mathcal{X}$.  A divisor $D$ on $\mathcal{X}(\mathbb{F}_q)$ is a formal integer sum of points on the curve, \[D=\sum_{P\in \mathcal{X}(\mathbb{F}_q)}a_P P,\] where $a_P\in\mathbb{Z}$ for all $P$. Divisors on $\mathcal{X}(\mathbb{F}_q)$ form a finite abelian group under formal addition.  The degree of $D$ is defined as $\textrm{deg}(D)=\sum_{P\in \mathcal{X}(\mathbb{F}_q)}a_P$.  A divisor $D$ with $a_P\geq 0$ for all $P$ is said to be effective, denoted by $D\geq 0$.  The support of the divisor $D$ is \[\textrm{supp}(D)=\{P:a_P\neq 0\}.\] 

For $f\in \F_q(\mathcal{X})$, define the divisor of $f$ to be \[\textrm{div}(f)=\sum_{P\in \mathcal{X}(\mathbb{F}_q)}\textrm{ord}_{P}(f) P,\] where $\textrm{ord}_P(f)$ is the order of vanishing of $f$ at the point $P$ and is negative when $f$ has a pole at point $P$.  A divisor of a function is called a principal divisor.  All principal divisors have degree 0, a very useful fact in bounding the minimum distance of the codes in this paper.

Given $D$, a divisor on $\mathcal{X}(\mathbb{F}_q)$, define the Riemann-Roch space 
\[\mathcal{L}(D)=\{f\in \F_q(\mathcal{X}): \textrm{div}(f)+D\geq 0\}\cup \{0\}.\]  The set $\mathcal{L}(D)$ is an $\mathbb{F}_q$-vector space, with dimension denoted $\ell(D)$.  The Riemann-Roch Theorem states that \[\ell(D)\geq \textrm{deg}(D) +1-g_{\mathcal{X}},\] where equality holds if $\textrm{deg}(D) \geq 2g_{\mathcal{X}}-1$.  This fact is useful in bounding the dimension of the codes in this paper.

\begin{figure}[h]
\begin{center}
\includegraphics[scale=.2]{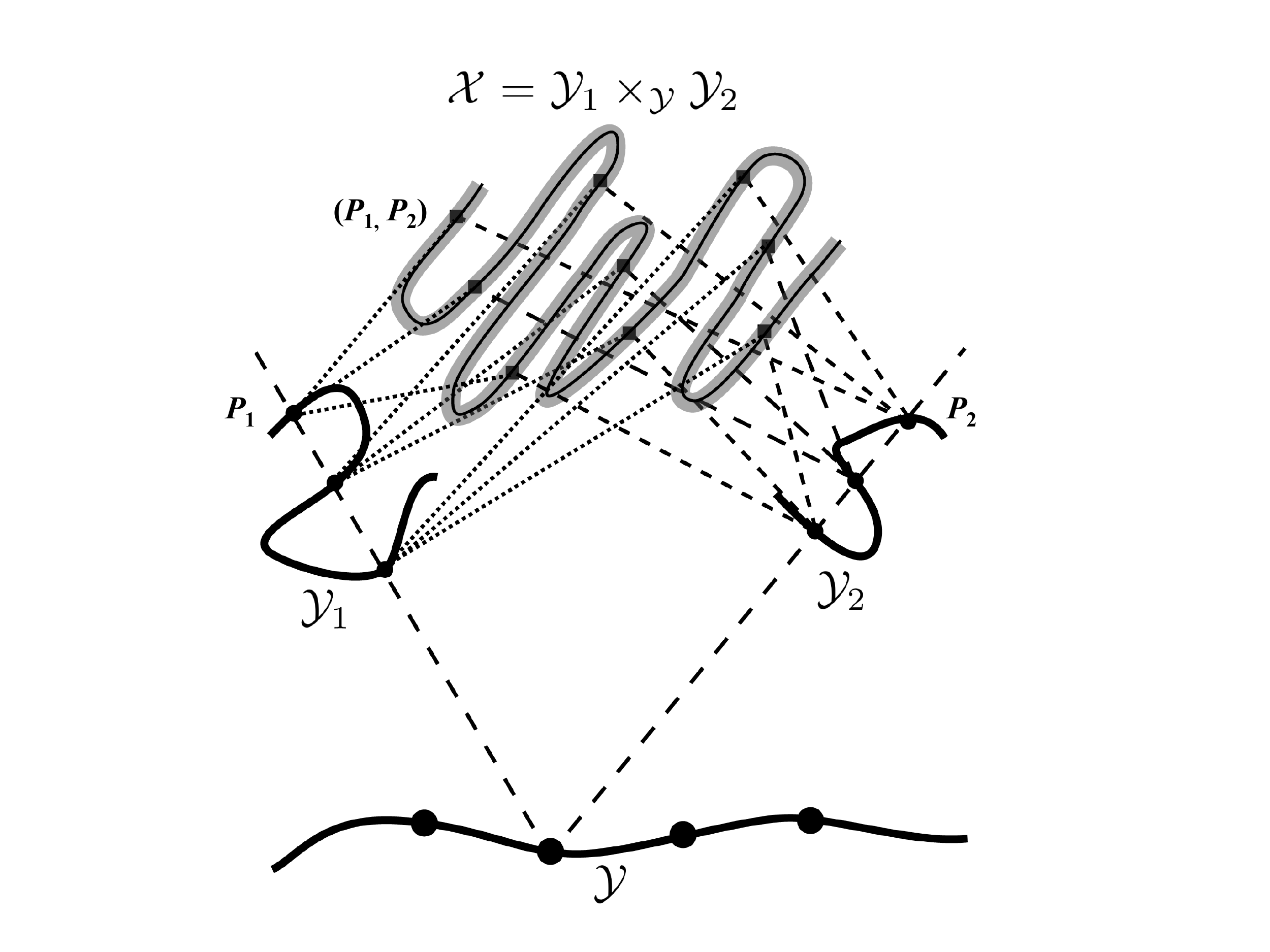}
\end{center}
\caption{A visualization of points on a fiber product of two curves. Points on the fiber product $\mathcal{X}$ may be thought of as tuples of points on the curves $\mathcal{Y}_1$ and $\mathcal{Y}_2$ which lie above the same point on $\mathcal{Y}$.}
\label{fp}
\end{figure}

The fiber product of curves is a geometric/combinatorial object that is at the center of the main code construction of this paper. Let $\mathcal{Y}, \mathcal{Y}_1, \mathcal{Y}_2, \ldots, \mathcal{Y}_t$ be smooth projective absolutely irreducible algebraic curves over $\mathbb{F}_q$ with  rational, separable maps $h_j: \mathcal{Y}_j \rightarrow \mathcal{Y}$.  The fiber product of $\mathcal{Y}_1, \mathcal{Y}_2, \ldots, \mathcal{Y}_t$ over $\mathcal{Y}$ is denoted by $\mathcal{Y}_1\times_{\mathcal{Y}}\mathcal{Y}_2\times_{\mathcal{Y}}\dots \times_{\mathcal{Y}}\mathcal{Y}_t$ and defined in the usual way as the universal object formed by the pullback of the maps $h_j$.  The fiber product of $t$ curves is also an algebraic curve.  Let $\mathcal{X}=\mathcal{Y}_1\times_{\mathcal{Y}}\mathcal{Y}_2\times_{\mathcal{Y}}\dots \times_{\mathcal{Y}}\mathcal{Y}_t$.  Concretely, the $\mathbb{F}_q$-rational points of the fiber product $\mathcal{X}$ are given by 
\[\mathcal{X}(\mathbb{F}_q)=\{(P_1,P_2,\dots, P_t):P_i\in\mathcal{Y}_i(\mathbb{F}_q) \textrm{ and }h_i(P_i)=h_j(P_j) \textrm{ for all }i,j, 1\leq i,j \leq t\}.\]  See Figure \ref{fp} for a visualization when $t=2$.

The fiber product construction defines natural projection maps $g_j:\mathcal{X}\rightarrow \mathcal{Y}_j$.  The degrees of the projection maps are given by 
\[\textrm{deg}(g_j)=\prod_{i\neq j}\textrm{deg}(h_i).\]  These maps have the property that $h_i\circ g_i=h_j\circ g_j$ for all $i,j$, so let $g:\mathcal{X}\rightarrow\mathcal{Y}$ be the composite map, with $\textrm{deg}(g)=\prod\textrm{deg}(h_j)$.  Thus, the diagram in Figure \ref{cd} commutes.

\begin{figure}[h] 
\begin{tikzpicture}[node distance = 2cm, auto] 
  \node (Y) {$\mathcal{Y}$};
  \node (Y1) [node distance=1.4cm, left of=Y, above of=Y] {$\mathcal{Y}_2$};
    \node (Ys) [node distance=1.4cm, right of=Y, above of=Y] {$\mathcal{Y}_t$};
   \node (X) [node distance=1.4cm, left of=Ys, above of=Ys] {$\mathcal{X}$};
 \node (Y2) [node distance=1.4cm, left of=Y1, left of=Ys, below of=X] {$\mathcal{Y}_1$};
  \draw[->] (Y1) to node {$h_2$} (Y);
   \draw[->] (Y2) to node [swap] {$h_1$} (Y);
  \draw[->] (Ys) to node {$h_t$} (Y);
   \draw[->] (X) to node [swap] {$g_1$} (Y2);
  \draw[->] (X) to node {$g_2$} (Y1);
  \draw[->] (X) to node {$g_t$} (Ys);
  \path (Y1) -- node[auto=false]{\ \ \ldots } (Ys);
  \draw[->, bend left] (X) to node {$g$} (Y);
\end{tikzpicture}
  \caption{The fiber product $\mathcal{X}$ of $t$ curves $\mathcal{Y}_j$. }
  \label{cd}
  \end{figure}
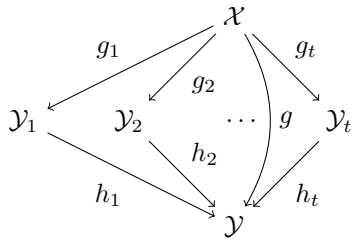
  
 Let $i\in[t]$.  Define the curve 
\[\tilde{\mathcal{Y}}_i= \mathcal{Y}_1\times_{\mathcal{Y}} \cdots \times_{\mathcal{Y}}\mathcal{Y}_{i-1}\times_{\mathcal{Y}}\mathcal{Y}_{i+1}\times_{\mathcal{Y}}\cdots \times_{\mathcal{Y}} \mathcal{Y}_t.\]
Then $\mathcal{X}=\mathcal{Y}_i\times_{\mathcal{Y}}\tilde{\mathcal{Y}}_i$. Denote the associated natural maps by
\[\tilde{g}_i: \mathcal{X} \rightarrow \tilde{\mathcal{Y}}_i \hspace{.25in} \textrm{ and } \hspace{.25in} \tilde{h}_i:\tilde{\mathcal{Y}}_i\rightarrow \mathcal{Y}.\] The degree of $\tilde{g}_i$ must be equal to the degree of $h_i$, denoted $d_{h_i}$.

The relationship of function fields in a fiber product is that $\mathbb{F}_q(\mathcal{Y})$ naturally embeds into $\mathbb{F}_q(\mathcal{Y}_i)$ for each $i$, and $\mathbb{F}_q(\mathcal{X})$ is isomorphic to the compositum of all the fields $\mathbb{F}_q(\mathcal{Y}_i)$.  In particular, fixing $i$, $1 \leq i \leq t$, then $g$,  $h_i$, and $g_i$ give rise to maps
$$
\begin{array}{cccc}
g^*: & \F_q(\mathcal{Y}) &\rightarrow & \F_q(\mathcal{X})\\
& f &\mapsto & f\circ g,
\end{array}
$$
$$
\begin{array}{cccc}
h_i^*: &\F_q(\mathcal{Y}) &\rightarrow &\F_q(\mathcal{Y}_i) \\
&f & \mapsto & f \circ h_i,
\end{array}
$$
and
$$
\begin{array}{cccc}
g_i^*: &\F_q(\mathcal{Y}_i) &\rightarrow &\F_q(\mathcal{X}) \\
&f & \mapsto & f \circ g_i.
\end{array}
$$
One may consider $g^*$,  $h_i^*$ and $g_i^*$ as embeddings so that we have
$$
\F_q(\mathcal{Y}) \xrightarrow{\sim} h_i^*(\F_q(\mathcal{Y})) \hookrightarrow \F_q(\mathcal{Y}_i)
 \xrightarrow{\sim} g_i^*(\F_q(\mathcal{Y}_i)) \hookrightarrow \F_q(\mathcal{X}).
$$

Then $\F_q(\mathcal{X})$ is the compositum of the fields $g_i^*\left(\F_q(\mathcal{Y}_{i})\right)$.  Also,  
\[g^*(\F_q(\mathcal{Y}))=\bigcap_{i=1}^t g_i^*(\F_q(\mathcal{Y}_i)).\]

Since any finite, separable algebraic extension of function fields is generated by a primitive element, there exists some function $x_i \in \mathbb{F}_q(\mathcal{Y}_i)$ with $\mathbb{F}_q(\mathcal{Y}_i)\cong h_i^*(\mathbb{F}_q(\mathcal{Y}))(x_i)$.  Thus we may construct the following diagram of function fields for a fiber product as in Figure \ref{cd2}.
\rmv{\begin{center}
\begin{tikzpicture}[node distance = 3cm, auto] 
  \node (Y) {$\mathbb{F}_q(\mathcal{Y})$};
  \node (X1) [node distance=2.5cm, left of=Y, above of=Y] {\begin{footnotesize}$h_1^*(\F_q(\mathcal{Y}))(x_1)$\end{footnotesize}};
    \node (Xs) [node distance=2.5cm, right of=Y, above of=Y] {\begin{footnotesize}$h_t^*(\F_q(\mathcal{Y}))(x_t)$\end{footnotesize}};
   \node (X) [node distance=2.5cm, left of=Xs, above of=Xs] {$\F_q(\mathcal{X})\cong g^*(\F_q(\mathcal{Y}))(h_1^*(x_1),... ,h_t^*(x_t))$};
 \node (X2) [node distance=2.5cm, right of=X1, left of=Xs, below of=X] {\begin{footnotesize}$h_2^*(\F_q(\mathcal{Y}))(x_2)$\end{footnotesize}};
  \draw[->]  (Y) to (X1);
   \draw[->] (Y) to  (X2);
  \draw[->] (Y) to  (Xs);
   \draw[->] (X1) to  (X);
  \draw[->] (X2) to (X);
  \draw[->] (Xs) to (X);
  \path (X2) -- node[auto=false]{\dots } (Xs);
\end{tikzpicture}
\end{center}}

\begin{figure}[h]
\begin{tikzpicture}[node distance = 3cm, auto] 
  \node (Y) {$\mathbb{F}_q(\mathcal{Y})$};
  \node (X1) [node distance=2.5cm, left of=Y, above of=Y] {\begin{footnotesize}$h_1^*(\F_q(\mathcal{Y}))(x_1)$\end{footnotesize}};
    \node (Xs) [node distance=2.5cm, right of=Y, above of=Y] {\begin{footnotesize}$h_t^*(\F_q(\mathcal{Y}))(x_t)$\end{footnotesize}};
   \node (X) [node distance=2.5cm, left of=Xs, above of=Xs] {$\F_q(\mathcal{X})\cong g^*(\F_q(\mathcal{Y}))(h_1^*(x_1),... ,h_t^*(x_t))$};
 \node (X2) [node distance=2.5cm, right of=X1, left of=Xs, below of=X] {\begin{footnotesize}$h_2^*(\F_q(\mathcal{Y}))(x_2)$\end{footnotesize}};
  \draw[->]  (Y) to (X1);
   \draw[->] (Y) to  (X2);
  \draw[->] (Y) to  (Xs);
   \draw[->] (X1) to  (X);
  \draw[->] (X2) to (X);
  \draw[->] (Xs) to (X);
  \path (X2) -- node[auto=false]{\dots } (Xs);
\end{tikzpicture}
 \caption{Function fields associated with the fiber product. }
  \label{cd2}
\end{figure}

\section{Construction of LRC($t$) } \label{construction_subsection}

Let $ \mathcal{Y}, \mathcal{Y}_1, \mathcal{Y}_2, \ldots, \mathcal{Y}_t$ be smooth, projective, absolutely irreducible algebraic curves over $\mathbb{F}_q$ with rational, separable maps $h_j: \mathcal{Y}_j \rightarrow \mathcal{Y}$, of degrees $d_{h_j}$ and let $\mathcal{X}$ be the fiber product $\mathcal{Y}_1\times_{\mathcal{Y}}\mathcal{Y}_2\times_{\mathcal{Y}}\dots \times_{\mathcal{Y}}\mathcal{Y}_t$.  As in Section \ref{Notation}, we obtain a projection map $g_j: \mathcal{X} \rightarrow  \mathcal{Y}_j$ of degree $d_{g_j}$ for  each $j$, $1 \leq j \leq t$, and a rational, separable map $g: \mathcal{X} \rightarrow \mathcal{Y}$ of degree $d_g$, so that the diagram in Figure \ref{cd} commutes. 

For each $j\in[t]$, let $x_j \in \F_q(\mathcal{Y}_j)$ be a primitive element of $\F_q(\mathcal{Y}_j)/h_j^*\left(\F_q(\mathcal{Y})\right)$ where $x_j$ is the root of a degree $d_{h_j}$ polynomial with coefficients in  $g_j^*(\F_q(\mathcal{Y}))$.  For convenience, denote $g_j^*(x_j)\in \F_q(\mathcal{X})$ by $x_j^*$.   Let $D_j$ be the divisor of the function $x_j^*\in \F_q(\mathcal{X})$.  Let $D_j=D_{j,+}-D_{j,-}$, where $D_{j,+}$ and $D_{j,-}$ are both effective; that is, $D_{j,+}=(x_j^*)_0$ is the zero divisor of $x_j^*$, and $D_{j,-}=(x_j^*)_{\infty}$ is the pole divisor of $x_j^*$.  Let $\textrm{deg}(D_{j,-})=d_{x_j}$ be the degree $x_j$ considered as  a function $\mathcal{Y}_j\rightarrow\mathbb{P}^1$.  Note that $d_{x_j}$ is not necessarily equal to $d_{h_j}$.  Then, if $x_j^*$ is viewed as a function $x_j^*:\mathcal{X}\rightarrow \mathbb{P}^1$, its degree is $d_{g_j}d_{x_j}$.  This yields \[g^*(\F_q(\mathcal{Y}))(x_1^*,x_2^*,\dots,x_t^*)\cong\F_q(\mathcal{X}).\]

Now choose simultaneously a divisor $D$ on $\mathcal{Y}(\F_q)$ and a set $S$ of points in $\mathcal{Y}(\F_q)$ as follows. Let $\tilde{D}=\sum_{i=1}^t h_i(D_{i,-})$, so $\textrm{supp}(\tilde{D})$ consists of all points of $\mathcal{Y}$ which, for some $i$, are the image under $h_i$ of a point on $\mathcal{Y}_i$ at which the function $x_i$ has a pole.  Then choose $D$ an effective divisor on $\mathcal{Y}(\F_q)$ of degree $\deg(D)=l$ and $S=\{P_1, \dots, P_s\} \subset \mathcal{Y}(\F_q)$ so that the following conditions are satisfied: 
\begin{itemize} 
 \item $|g^{-1}(P_i)\cap\mathcal{X}(\F_q)|=d_g,$ for all $i\in[s]$, 
 \item $S \cap \textrm{supp} (\tilde{D}) = \emptyset$,  
 \item $S\cap \textrm{supp}(D) = \emptyset$, 
 \item  $l< s$. 
 
 \end{itemize} 
Let $\{f_1, f_2,\dots, f_m\}$ be a basis for $\mathcal{L}(D)$, so $m=\ell(D)$. Note that since $l<s$, each non-zero function in $\mathcal{L}(D)$ will be non-zero when evaluated at some point in $S$. 
Then set 
$$B=g^{-1}(S).$$
Note that $n=|B|=d_gs$. Order the points in $B$ and denote them as $\{Q_1, Q_2, \ldots, Q_n\}$. For $Q_i \in B$, set
$$B^{(j)}(Q_i)=\tilde{g}_j^{-1}(\tilde{g}_j(Q_i))$$ and 
$$A^{(j)}(Q_i)=B^{(j)}(Q_i) \setminus \{ Q_i \}.$$

The recovery set $A_{ij}$ defined in Section~\ref{Notation} can be obtained from $A^{(j)}(Q_i)$ as follows: 

 \[ A_{ij} = \{k\in [n]: Q_k \in A^{(j)}(Q_i)\}.\] 
 Similarly,  $B_{ij} =  \{k\in [n]: Q_k \in B^{(j)}(Q_i)\}$. 

For $k\in [m]$, let $f_k^*=g^*(f_k)$. Then let 
\[ V=\textrm{Span}\{ f_k^{*}x_1^{*{e_1}}\cdots {x_t^*}^{e_t}: e_i\in\mathbb{Z}, 0\leq e_i\leq d_{h_i}-2\textrm{ for all } i\in [t], k\in [m] \}. \] 
Define
$C(D,B):=Im (ev_B)$ where 
$$
\begin{array}{llll}
ev_B: &V &\rightarrow &\F_q^{n} \\
&f& \mapsto & \left( f \left( Q_i\right) \right)_{i\in [n]}.
\end{array}
$$

%


This construction gives rise to the following theorem.
\begin{theorem} \label{construction}
Given curves $\{\mathcal{Y}_i\}_{i\in[t]}$, $\mathcal{Y}$, maps $\{h_i:\mathcal{Y}_i \rightarrow\mathcal{Y}\}_{i\in[t]}$, a divisor $D$ on $\mathcal{Y}(\F_q)$, and sets $S\subset\mathcal{Y}(\F_q)$ and $B=g^{-1}(S)$ all as described above, where $l=\textrm{deg}(D)\leq |S|$ and the quantity $d$ below is positive, the code $C(D,B)$ is an LRC($t$) with 
\begin{itemize}
\item length $n=|B|=d_g|S|$, 
\item dimension $$
\begin{array}{lll}
k&=&\ell(D)(d_{h_1}-1)(d_{h_2}-1) \cdots (d_{h_t}-1) \\
&\geq& (l-g_{\mathcal{Y}}+1)(d_{h_1}-1)(d_{h_2}-1) \cdots (d_{h_t}-1),
\end{array}$$
\item minimum distance $d \geq n-ld_g-\sum_{i=1}^t\left( d_{h_i}-2\right)\left(d_{g_i}d_{x_i}\right)$, and 
\item locality $(d_{h_1}-1, d_{h_2}-1, \dots, d_{h_t}-1)$.
\end{itemize}
\end{theorem}

\begin{proof}
The length of  $C(D,B)$ is given by definition.  The minimum distance is bounded by considering the number of zeros that a function $f\in V$ can have on the curve $\mathcal{X}$.  Note that 
\[f_k^*x_1^{*e_1}x_2^{*e_2}\cdots x_t^{*e_t}\in\mathcal{L}\left(g^{-1}(D)+\sum_{j=1}^t\left(d_{h_j}-2\right)g_j^{-1}(D_{j,-})\right)\subset\mathbb{F}_q(\mathcal{X}),\] where $D$ has degree $l$, $g^{-1}(D)$ has degree $l d_g$, and $g_j^{-1}(D_{j,-})$ has degree equal to $d_{g_j}d_{x_j}$. Thus the function $f$ can have at most $ld_g+\sum_{j=1}^t(d_{h_j}-2)d_{g_j}d_{x_j}$ poles on $\mathcal{X}$ (counted with multiplicity). As any principal divisor has degree 0, $f$ may have at most this many zeros on $\mathcal{X}$ (counted with multiplicity). The dimension is as stated, because the evaluation map is injective for $e_i$ in the ranges described.  

Next, we see how local recovery is achieved. Say that $\mathbf{c}=(c_1,\dots, c_n)\in C(D,B)$.  Fix  $j\in [t]$.  Let 
\[I=\{\mathbf{e}=(e_1,e_2,\dots,e_t):e_i\in \mathbb{Z}, 0\leq e_i\leq d_{h_i}-2 \textrm{ for all } i\},\] and \[I_j=\{\tilde{\mathbf{e}}=(e_1,e_2,\dots,e_{j-1}, e_{j+1}, \dots, e_t):e_i\in \mathbb{Z}, 0\leq e_i\leq d_{h_i}-2 \textrm{ for all } i\neq j\}.\] 

Fix $i\in[n]$. To use the $j$th recovery set to recover the value $c_i$, consider that for some $f\in V$, $c_i=f(Q_i)$.  Recall that $m=\ell(D)$. Say 
\[f=\sum_{k=1}^m a_{k}f_k^{*}\sum_{\mathbf{e}\in I} b_{k,\mathbf{e}}\left(\prod_{w\in[t]}{x_{w}^*}^{e_w}\right),\] where $a_k,b_{k,\mathbf{e}}\in\F_q$ are constants depending on $f$. Rearranging the sum, 
\[f=\sum_{e_j=0}^{d_j-2}b_{e_j}{x_j^*}^{e_j}\sum_{k=1}^m a_{e_j,k}f_k^{*}\sum_{\tilde{\mathbf{e}}\in I_j}b_{e_j,k,\tilde{\mathbf{e}}}\left(\prod_{w\in[t], w\neq j}{x_w^{*}}^{e_w}\right),\] 
where $b_{e_j}, a_{e_j,k}, b_{e_j,k,\tilde{\mathbf{e}}}\in\F_q$ are new constants depending on $f$.
Then
\[c_i=f(Q_i)=\sum_{e_j=0}^{d_j-2}b_{e_j}({x_j^*(Q_i)})^{e_j}\sum_{k=1}^m a_{e_j,k}f_k^{*}(Q_i)\sum_{\tilde{\mathbf{e}}\in I_j}b_{e_j,k,\tilde{\mathbf{e}}}\left(\prod_{w\in[t], w\neq j}({x_w^{*}(Q_i)})^{e_w}\right).\]
By definition this is just
\[\sum_{e_j=0}^{d_j-2}b_{e_j}({x_j(g_j(Q_i))})^{e_j}\sum_{k=1}^m a_{e_j,k}f_k(g(Q_i))\sum_{\tilde{\mathbf{e}}\in I_j}b_{e_j,k,\tilde{\mathbf{e}}}\left(\prod_{w\in[t], w\neq j}({x_w(g_w(Q_i))})^{e_w}\right).\]

For all $w\in[t]$, $w\neq j$, let $x_w(g_w(Q_i))=\beta_{w}\in\F_q$.  Define $f_k(g(Q_i))=\alpha_k\in\F_q$.  Putting this together, \[c_i=\sum_{e_j=0}^{d_j-2}b_{e_j}(x_j^*(Q_i))^{e_j}\sum_{k=1}^m a_{e_j,k}\alpha_k\sum_{\tilde{\mathbf{e}}\in I_j}b_{e_j,k,\tilde{\mathbf{e}}}\left(\prod_{w\in[t], w\neq j}(\beta_w)^{e_w}\right).\] Combining constants and redefining $b_{e_j}$ appropriately, \[c_i=\sum_{e_j=0}^{d_j-2}b_{e_j}(x_j^{*}(Q_i))^{e_j}.\] In this framework,  $c_i$ can be seen as the evaluation at $x_j^*(Q_i)$ of a polynomial in $x$:  
\[\tilde{f}(x):=\sum_{e_j=0}^{d_j-2}b_{e_j}(x)^{e_j},\] i.e., $c_i= f(Q_i)=\tilde{f}(x_j^{*}(Q_i))$. 

For all $Q\in B^{(j)}(Q_i)$, the definition of $ B^{(j)}(Q_i)$ implies that $\tilde{g}_j(Q)=\tilde{g}_j(Q_i).$  That means 
\begin{eqnarray}\label{first}
g_w(Q)=g_w(Q_i)
\end{eqnarray} for all $w\in [t]$, $w\neq j$.  Composing with $\tilde{h}_j$ in  (\ref{first}), we have 
\begin{eqnarray}\label{second}
g(Q)=\tilde{h}_j(\tilde{g}_j(Q))=\tilde{h}_j(\tilde{g}_j(Q_i))=g(Q_i).
\end{eqnarray} Then composing with $f_k$ for $k\in[m]$ in (\ref{second}), we have $f_k(g(Q))=f_k(g(Q_i))$.  Similarly, composing with the function $x_w$ in (\ref{first}) yields $x_w(g_w(Q_i))=x_w(g_w(Q))$ for $w\in[t]$ with $w\neq j$.  That means if $r\in A_{ij}$, then \[c_r=f(Q_r)=\sum_{e_j=0}^{d_j-2}b_{e_j}(x(Q_r))^{e_j}=\tilde{f}(x_j^*(Q_r)).\] Thus $c_r$ is the evaluation of $\tilde{f}$ at $x=x_j^*(Q_r)$.  Since it must be that there are no two points $Q_{a}\neq Q_{b}$ in $B^{(j)}(Q_i)$ where $x_j^*(Q_a)=x_j^*(Q_b)$, the values of $c_r$ for all $r\in A_{ij}$ are the evaluations of $\tilde{f}$ at $d_{h_j}-1$ distinct values in $\F_q$.  Thus the polynomial $\tilde{f}$ may be interpolated to determine $c_i=\tilde{f}(x_j^*(Q_i))$. 

\end{proof}

\begin{remark}
For $\rho\geq1$ an integer, let 
\[V_{\rho}=\textrm{Span}\{ f_k^{*}x_1^{*{e_1}}\cdots {x_t^*}^{e_t}: 0\leq e_j\leq d_{h_j}-1-\rho \textrm{ for all }j\in[t], k\in[m] \}. \] As observed in \cite{Barg16}, applying the same construction with $V_{\rho}$ in place of $V$ allows up to $\rho$ erasures to be recovered by each recovery set, increasing the local distance of $C(D,B)$.  However, this also reduces the dimension of the code to $$m(d_{h_1}-\rho)(d_{h_2}-\rho) \cdots (d_{h_t}-\rho),$$ so this must be seen as a tradeoff between dimension and recovery.  Throughout the rest of the paper, we use the maximal dimension construction from Theorem \ref{construction}, but the modification above may be made to recover more erasures if desired.
\end{remark}

Given curves $\{\mathcal{Y}_j\}_{j\in[t]}$ and $\mathcal{Y}$, as well as appropriate maps $\{h_j\}_{j\in[t]}$, Theorem \ref{construction} describes a code on the fiber product where knowledge of the parameters is built strictly on the understanding of these maps.  This may be considered as a `bottom-up' approach to building LRC($t$)s.  It is sometimes also possible to take a `top-down' approach using the automorphism group of a curve $\mathcal{X}$ to create maps to quotient curves $\mathcal{Y}_j$. This approach is described in the following corollary. 

\begin{corollary} \label{aut_gp_construction}
Suppose that $\mathcal{X}$ is a curve such that $Aut(\mathcal{X})$ has subgroups $T_1, \dots, T_t$ so that the subgroup generated by these groups is an associative semi-direct product within Aut($\mathcal{X}$). Then Theorem \ref{construction} applies to give locally recoverable codes from $\mathcal{X}$ by taking the $\mathcal{Y}_j$ to be the the associated quotient curves $\mathcal{Y}/T_j$  and  setting $\mathcal{Y}:=\mathcal{X}/(T_1\rtimes\dots \rtimes T_t)$.
\end{corollary}

\begin{remark} \label{group bad remark} The semi-direct product must be associative so that the curve $\mathcal{Y}$ is well defined.  For $t>2$, the requirement that the semi-direct product is associative is far from trivial. Another difficulty with the 
`top-down' approach of Corollary \ref{aut_gp_construction} is explicitly writing down equations for $\mathcal{Y}_j$, the functions $x_j$, and the spaces $\mathcal{L}(D)$ when starting with a given model of $\mathcal{X}$.  \end{remark}

\begin{remark} \label{differences remark}
Our construction builds on that of \cite[Theorem 5.2]{Barg16}, where the authors construct LRC(2)s from fiber products and bound the parameters of the resulting codes. In addition to carrying out the generalization to LRC($t$)s for $t\geq 2$, the bounds in Theorem \ref{construction} differ in some key ways.  First, \cite[Theorem 5.2]{Barg16} relies on the parameter $h$, which is the degree (as a function to $\mathbb{P}^1$) of a primitive element $x$ generating the full extension $\F_q(\mathcal{X})/g^*(\F_q(\mathcal{Y}))$.  This primitive element is not canonical.  It is known that this primitive element can be generated as a linear combination of the primitive elements $x_i^*$ of the extensions $\F_q(\mathcal{Y}_i)/h_i^*(\F_q(\mathcal{Y}))$, but not every linear combination results in a primitive element.  Also, the degree $h$ may vary depending on the the primitive element chosen.  The approach in Theorem \ref{construction} entirely avoids the issue of a primitive element for $\F_q(\mathcal{X})/g^*(\F_q(\mathcal{Y}))$ and relies only on the primitive elements for the intermediate extensions.
\end{remark}

As mentioned in the Section \ref{Introduction}, the curves considered in this paper are maximal.  Precisely, this means their numbers of rational points meet the upper Hasse-Weil bound, so $|\mathcal{C}\left(\mathbb{F}_q\right)| = q+1+2g_{\mathcal{C}}\sqrt{q}$, where $g_{\mathcal{C}}$ is the genus of the curve $\mathcal{C}$.  Maximal curves have played an important role in coding theory, especially in the construction of algebraic geometry codes, as they support the construction of long codes over relatively small fields, with minimum distance bounded below based on geometric arguments. Their utility in the construction of locally recoverable codes comes from this as well.  Generally, if $|\mathcal{X}(\F_q)|$ is large, then the hope is that one can find appropriate map $g:\mathcal{X} \rightarrow \mathcal{Y}$ so that $|B|$ is large. Focusing on the intermediate curves, if $|\mathcal{Y}_j(\F_q)|$ is large for each $j$, and the maps $h_j$ have limited and aligned ramification, the set $|B|$ can also be large.  Moreover, some families of maximal curves (such as Hermitian and Suzuki) have large automorphism groups which may allow for more choices when constructing recovery sets (equivalently the maps $g_j:\mathcal{Y}_j \rightarrow \mathcal{Y}$).


\section{LRC($t$) codes, erasure recovery, and error correction} \label{recovery_issues}

As described in Section \ref{Introduction}, locally recoverable codes have been developed with the primary goal of facilitating convenient recovery of erasures, potentially created by server failure in distributed storage systems.  The construction for $C(D,B)$ described in Section \ref{construction_subsection} results in $t$ disjoint recovery sets for each location. 
Since recovery sets are defined by fibers over points of the curves $\tilde{\mathcal{Y}_j}$, the position $i$ is in the $j$th recovery set for the position $k$  if and only if position $k$ is in the $j$th recovery set for $c_i$; i.e., for all $j \in [t]$,
$$ k\in A_{ij} \Leftrightarrow i\in A_{kj}.$$

We say that such recovery sets are symmetric.  Note that for each $j\in [t]$ there exists some $I\subseteq [n]$ such that $\{B_{ij}\}_{i\in I}$ is a partition of $[n]$. 

This section addresses two questions that arise from the recovery procedure for symmetric recovery sets.  First, what benefit is provided by having a large number of recovery sets?  Second, how does the number of recovery sets relate to minimum distance and global error correction?

First, we note that for erasure of a small fraction of locations in known positions, it can be much more efficient to use local recovery over global error correction, even if multiple local recovery sets are required.  Consider a locally recoverable code $C$ of length $n$ and availability $t$ with symmetric recovery sets.  Say that $\epsilon$ erasures occur.  For global error correction, we assume that all $n-\epsilon$ known locations must be consulted.  In general, we see that if a code has locality $t$ and symmetric recovery sets, then the fact that recovery sets are transverse implies that $lcm \left\{ (r_j+1): 1 \leq j \leq t \right\}$ divides $n$. Let $m_t(\epsilon)$ be the maximum number of positions that need to be consulted to recover $\epsilon$ erasures using local recovery. If there is a pattern of $\epsilon$ erasures that may not be locally recoverable by a code of availability $t$, then $m_t(\epsilon)=\infty$.  

If only one position $c_P$ has been erased and $C$ admits a single local recovery set of size $r_1$ for each position, then $r_1$ positions may be consulted to recover $c_P$, and $m_1(1)\leq r_1$.  We may assume $n=s(r_1+1)$ for some integer $s>1$, so local recovery saves at least $(s-1)(r_1+1)$ consultations. 
 
Now consider the possibility that two positions $c_P$ and $c_Q$ were erased.  If $C$ has only one recovery set for each position and the recovery set of $c_P$ contains $c_Q$ (and vice versa, by symmetry), then $c_P$ and $c_Q$ are not recoverable using that recovery set.  However, if $C$ has availability $2$, then $C$ has two disjoint recovery sets for each position, and the fact that the first recovery set of $c_P$ contains $c_Q$ implies that the second recovery set does not (and similarly for $c_Q$).  If $C$ has locality $(r_1, r_2)$, then $c_P$ and $c_Q$ can both be recovered with at most $r_1+r_2$ consultations.  We have saved at least $(s(r_1+1)-2)-(r_1+r_2)$ consultations.  

Without loss of generality, assume $r_1=\max_{1\leq j\leq t}\{r_j\}$.  For a general code of availability $t$, assuming that $t$ is sufficiently large to recover $\epsilon$ erasures, recovery will require at most $\epsilon$ recovery sets be consulted, resulting in a total number of 
\[m_t(\epsilon)\leq \epsilon r_1\] consultations in local recovery, in comparison with at least $s(r_1+1)-\epsilon$ consultations for global error correction.  This also assumes that the minimum distance of the code is sufficiently large to correct $\epsilon$ errors, which is not necessarily implied by the construction. Clearly, if $s$ is close in size to $\epsilon$, then there is little difference in the numbers of consultations required for local recovery and for error correction. However, our construction generally results in relatively large $s$. In all our examples, we have that $\prod_{j=1}^{t}(r_j+1) $ divides $n$, so the savings are significant for moderate size $\epsilon$.

In fact, the interplay between parameters becomes complicated as the number of potential failures grows. The following result demonstrates that more than two recovery sets become necessary rather quickly.
\begin{proposition}\label{neededlocality}
Let $b(\epsilon)$ be the availability required for a locally recoverable code $C$ with symmetric recovery sets to be capable of locally recovering any pattern of $\epsilon$ erasures. Then
\begin{itemize}
\item $b(1)=1$,
\item $b(2)=2$,
\item $b(3)=2$, and
\item $b(4)>3$.
\end{itemize} 

\end{proposition}
\begin{proof}
 
It is not difficult to see that $b(1)=1$ and $b(2)=2$ based on the situations described above.  To see that $b(3)=2$, assume without loss of generality that positions 1, 2, and 3 have been erased. Availability $t=1$ is clearly not sufficient, because it could be that $B_{11} = B_{21} = B_{31}$, in which case no $A_{i1}$ is complete. However, $t=2$ is sufficient; if $B_{11} = B_{21} = B_{31}$, then 
\[ B_{12}\cap B_{22}= B_{12}\cap B_{32} = \emptyset, \]
since distinct recovery sets for a fixed position are disjoint. Thus, position 1 may be recovered from the recovery set $A_{12}$, leaving two erased positions, which can be recovered with two recovery sets since $b(2)=2$. If $B_{11}\neq B_{21}$ but $B_{11} = B_{31}$, then position 2 can be recovered from $A_{21}$, again leaving two erased positions and two recovery sets. 

To see that $b(4)>3$, consider the following scenario.  WLOG, assume that positions $1, 2, 3,$ and $4$ have been erased, and suppose that: 
\begin{itemize}
\item $B_{11} = B_{21}$ and $B_{31}=B_{41}$,
\item $B_{12} = B_{32}$ and $B_{22} = B_{42}$,
\item $B_{13} = B_{43}$ and $B_{23} = B_{33}$. 
\end{itemize}
Clearly, none of the three recovery sets can be used to recover any of the erased positions.  Therefore $b(4)>3$.

\end{proof}

%
%
%
%
%
%

Note that the recovery procedure and all discussion to this point assume that the positions of the erased locations are known. This is not the case in general error correction, so a code that is capable of restoring $\epsilon$ erasures may not be capable of correcting $\epsilon$ global errors.  Minimum distance bounds give some indication of the global error correction capability of the code.  The minimum distance bounds in Theorem \ref{construction} are based on the divisors of certain functions.  For theoretical purposes, however, we would like to bound the minimum distance without explicitly constructing the given field extensions as in Section \ref{construction_subsection}. More generally, one might wish to know if an LRC($t$) is actually an error-correcting code, regardless of its construction. It is interesting to consider how the presence of recovery sets influences the potential error-correcting capability of a locally recoverable code, independent of how the code itself is defined. With this in mind, some very modest bounds on minimum distance can be derived by considering that, by construction, there exists an algorithm for correcting a certain number of erasures.

\begin{proposition}\label{distance}
Let $C$ be a locally recoverable code of availability $t$ with symmetric recovery sets and minimum distance $d$.  Then 
\begin{itemize}
\item if $t\geq 1$, then $d\geq 2$,
\item if $t\geq 2$, then $d\geq 3$.

\end{itemize}
\end{proposition}
\begin{proof}
Suppose that $t=1$.  Assume a codeword $\mathbf{c}$ has one error, in position $i$, resulting in the word $\mathbf{c}^{\prime}$.  Consider an experiment in which a single position is intentionally erased and is locally recovered using the recovery set for the position.  If there were no errors, any position could be erased and correctly recovered from its recovery set.  However, if $c_i$ is erased, the recovered word will be different than $\mathbf{c}^{\prime}$, meaning that there must be an error in $B_{i1}$.  Checking all positions using this procedure, one may detect any single error in $\mathbf{c}$, so the minimum distance of $C$ is at least 2.

Suppose that $t=2$.  Assume a codeword $\mathbf{c}$ has one error, in position $i$, resulting in the word $\mathbf{c}^{\prime}$.  Consider an experiment in which a single position is intentionally erased, and is then locally recovered from each of its two disjoint recovery sets.  Again, if there were no errors, any position could be erased and correctly recovered  from either of its two recovery sets.  However, if $c_i$ is erased, the words recovered using $A_{i1}$ and $A_{i2}$ will  \textit{both} be different from $\mathbf{c}^{\prime}$, meaning that there must be an error in $B_{i1}$ and in $B_{i2}$.  Since there is only one error in $\mathbf{c}^{\prime}$, the single error must be at the intersection of $B_{i1}$ and $B_{i2}$, namely at position $i$.  Once the position of the error is known, it may be erased and correctly recovered using either of the two recovery sets.   Moreover, none of the other positions will register errors from both of their recovery sets, so position $i$ can be uniquely identified.

\end{proof}

\section{LRC$(2)$s on generalized Giulietti-Korchmaros curves} \label{GK_section}

Let $q=p^h$ for $p$ a prime, and let $N\geq 3$ be an odd natural number. We consider the family of generalized Giulietti-Korchmaros (GK) curves $\mathcal{C}_N$, which are maximal over the field $\mathbb{F}_{q^{2N}}$ \cite{Garcia, GMP}.  The curve $\mathcal{C}_N$ is the normalization of the intersection of two surfaces $\mathcal{H}_q$ and $\mathcal{X}_N$ in $\mathbb{P}^3$, defined by the following affine equations: 
\[ \mathcal{H}_q: x^q+x=y^{q+1} \] 
\[ \mathcal{X}_N: y^{q^2} - y =z^{\frac{q^N+1}{q+1}}. \]  
The intersection of these surfaces has a single point at infinity, denoted by $\infty$, which is a cusp singularity for $N>3$ and is smooth when $N=3$.  The curve is smooth elsewhere. 
The number of $\mathbb{F}_{q^{2N}}$-rational points on $\mathcal{C}_N$ is 
\[ \#\mathcal{C}_N(\mathbb{F}_{q^{2N}} )= q^{2N+2} - q^{N+3} + q^{N+2} + 1. \]

The curve $\mathcal{C}_N$ can also be defined as the normalized fiber product over $\mathbb{P}^1$ of the two curves in $\mathbb{P}^2$ given by the same equations.  As curves, both $\mathcal{H}_q$ and $\mathcal{X}_N$ are maximal over the field $\mathbb{F}_{q^{2N}}$ \cite{Abdon}, each with a single point at infinity, denoted by $\infty_{\mathcal{H}_q}$ and $\infty_{\mathcal{X}_N}$ respectively.  Let $\infty_{y}$ denote the single point at infinity on $\mathbb{P}^1_y$.  Define $h_1:\mathcal{X}_N\rightarrow \mathbb{P}^1_y$ to be the natural degree $\frac{q^N+1}{q+1}$ projection map onto the $y$ coordinate for affine points, with $\infty_{\mathcal{X}_N}\mapsto \infty_{y}$.  Similarly, let $h_2:\mathcal{H}_q\rightarrow \mathbb{P}^1_y$ be the natural degree $q$ projection map onto the $y$ coordinate for affine points, with $\infty_{\mathcal{H}_q}\mapsto \infty_{y}$. We then have the fiber product construction depicted in Figure \ref{GK_fig}.

\begin{figure} 
\begin{tikzpicture}[node distance = 2cm, auto]
  \node (Y) {$\mathbb{P}^1_y$};
  \node (Y1) [node distance=1.4cm, left of=Y, above of=Y] {$\mathcal{X}_N$};
   \node (Y2) [node distance=1.4cm, right of=Y, above of=Y] {$\mathcal{H}_q$};
   \node (X) [node distance=1.4cm, left of=Y2, above of=Y2] {$\mathcal{X}_N\times\mathcal{H}_q$};
  \draw[->] (Y1) to node [swap] {$h_1$} (Y);
  \draw[->] (Y2) to node {$h_2$} (Y);
  \draw[->] (X) to node [swap] {$g_1$} (Y1);
  \draw[->] (X) to node {$g_2$} (Y2);
  \draw[->] (X) to node {$g$} (Y);
 \end{tikzpicture}
   \caption{Generalized GK curve as a fiber product.}
    \label{GK_fig}
\end{figure}
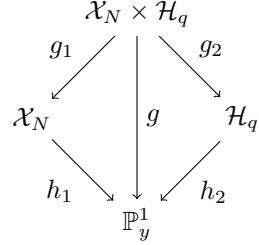

Then $\mathcal{C}_N=\widetilde{\mathcal{X}_N\times\mathcal{H}_q}$, the normalization of the fiber product described above.  The map $g:\mathcal{C}_N\rightarrow \mathbb{P}^1_y$ has degree $d_g:=\frac{q(q^N+1)}{q+1}$ and is ramified above $\infty_y$ and $a_y$ with $a \in \F_{q^2}$, where $a_y$ denotes the point on $\mathbb{P}^1_y$ with $y=a$  \cite{Garcia, GMP}. 
Notice that $d_{g_1}=d_{h_2}=q$ and $d_{g_2}=d_{h_1}=\frac{q^N+1}{q+1}$.
 
To construct an LRC($2$), we use the commutative diagram in Figure \ref{GK_fig} and the construction detailed in Section \ref{construction_subsection}. The degree of the function $x:\mathcal{H}_q\rightarrow\mathbb{P}^1$ is $d_x=q+1$.  The degree of the function $z:\mathcal{X}_N\rightarrow \mathbb{P}^1$ is $d_z=q^2$.  We take the divisor $Q$ to be $\infty_y$ and choose a parameter $l$ so that $D=l\infty_y$. 

\begin{theorem} The locally recoverable code $C(l \infty_y, B)$ constructed from the generalized GK curve $C_N$ as described above is an $[n,k,d]$ code over $\F_{q^{2N}}$ with availability $2$ and locality $\left(q-1,\frac{q^N+1}{q+1}-1 \right)$ where
\begin{align*}
n & = q^{2N+2} - q^{N+3} + q^{N+2} - q^3, \\
k &=\left(\frac{q^N+1}{q+1}-1\right)(q-1)(l+1), \\
d & \geq n - l\left(\frac{q(q^N+1)}{q+1}\right) - \left((q^N+1)\left(q-2\right)  +q^3\left(\frac{q^N+1}{q+1}-2\right)\right), \\
\end{align*}  
and $l$ is any positive integer yielding $0<k<n$,  $l<q^{N+2}+q^{N+1}-q-1$, and $d>0$.
\end{theorem}
 
 \begin{proof}
The set $B$ of evaluation points consists of points on the curve $\mathcal{C}_N$ that are not above ramification points in the ground curve, $\mathbb{P}^1_y$, over the field $\mathbb{F}_{q^{2N}}$.  This means we let $S=g(B)$. Since \[|g^{-1}(\{\infty_y, a_y: a \in \F_{q^2} \})|=q^3+1,\] we calculate the size of the evaluation set $B$: 
\begin{align*}
 |B| = | \mathcal{C}_N(\F_{q^{2N}})| - (q^3+1) &=   q^{2N+2} - q^{N+3} + q^{N+2} + 1 - (q^3 + 1)\\
  &= q^{2N+2} - q^{N+3} + q^{N+2} - q^3=sd_g, \end{align*} where $s=|S|=q^2(q^{N-1}-1)(q+1)=q^{N+2}+q^{N+1}-q-1$.

Let $D=l\infty_{y}$.  Since the genus of $\mathcal{Y}=\mathbb{P}^1$ is 0, we know by the Riemann-Roch Theorem that $\ell(D)=l+1$, and we can realize these functions as polynomials in $y$ of degree bounded by $l$.  The set of evaluation functions for the code is denoted by $V$, where 
\[ V= \text{Span}\left\{ x^i z^j y^{\kappa}: 0\leq i \leq \frac{q^N+1}{q+1}-2, 0\leq j\leq q-2, 0\leq \kappa \leq l, \textrm{ and }i,j,\kappa \in\mathbb{Z}\right\}. \]  
Then, by Theorem \ref{construction}, we obtain a code with the claimed attributes.
\end{proof}

\begin{remark}
As discussed in \cite{GMP}, the curves $\mathcal{X}_N$ and $\mathcal{H}_q$ are quotient curves of $\mathcal{C}_N$, and in this case the automorphism group construction in Corollary \ref{aut_gp_construction} could be applied, yielding the same codes.  
\end{remark}

Given the choice of $l$, we can construct codes with rate close to 
one, where the tradeoff is low minimum distance:

\[ R_{N,l}\geq \frac{\left(\frac{q^N+1}{q+1}-1\right)(q-1)(l+1)}{|B|}, \]
where for maximal $l$ and increasing $N$ we have
\[\lim_{N\rightarrow \infty}R_{N,l}=1.\]

\begin{example} 
Taking $N=3$, we  consider codes from the curves $\mathcal{C}_3$ over the field $\mathbb{F}_{q^6}$. We obtain blocklength 
$|B| = q^3(q-1)(q^2-q+1)(q+1)^2=q^8-q^6+q^5-q^3$, as above, and the following bounds on the dimension and  minimum distance: 
\begin{align*}
k&= (q^2-q)(q-1)(l+1), \\
d&\geq n-lq(q^2-q+1)-(q^3+1)(q-2)-q^3(q^2-q-1). 
\end{align*}

In Table~\ref{table:gk3}, we consider $q=3$ and provide bounds on the code parameters for different values of $l$. 

\begin{table}[h!] 
\centering
\begin{tabular}{|l|l|l|l|}
\hline
$l$ &  $k$ & $d\geq$ \\ \hline
270 &       3252    &    215  \\ \hline
260 &     	3132   &   425 \\ \hline
250 &          3012 &   635  \\ \hline
240 &         2892  &    845 \\ \hline
230 &        2772   &   1055  \\ \hline
220 &         2652  &   1265  \\ \hline
210 &         2532  &    1475 \\ \hline
\end{tabular}
\caption{The generalized GK curves $\mathcal{C}_3$ over $\mathbb{F}_{729}$ produce LRC(2)s 
 of length $n=6048$, with $N=3$, $q=3$, $r_1=6$, $r_2 = 2$, and $D=l\infty_y$, with $l$ determining $k$ and $d$ as above. }
 \label{table:gk3}
\end{table}

\end{example} 

In this section, we employed generalized GK curves to obtain LRC(2)s over $\F_{q^{2N}}$ with recovery sets of sizes $q-1$ and $\frac{q^N+1}{q+1}-1$. While this addresses the availability problem, it leads to recovery sets of very different sizes if $N>2$. In the next section, we construct LRC(2)s with recovery sets which are more balanced in size.

\section{LRC($2$)s on Suzuki curves} \label{Suzuki_section}

Let $a\in\mathbb{N}$, $q_0=2^a$, and $q=2q_0^2$.  The Suzuki group $Sz(q)$ can be realized as a subgroup of $GL_4(q)$ as follows.  Let $a,c,d\in\mathbb{F}_q$, $d\neq 0$, and define \[T_{a,c}=\left(\begin{array}{cccc}1 & 0 & 0 & 0 \\a & 1 & 0 & 0 \\c & a^{2q_0} & 1 & 0 \\a^{2q_0+2}+ac+c^{2q_0} & a^{2q_0+1}+c & a & 1\end{array}\right),\] \[M_d=\left(\begin{array}{cccc}d^{-q_0-1} & 0 & 0 & 0 \\0 & d^{-q_0} & 0 & 0 \\0 & 0 & d^{q_0} & 0 \\0 & 0 & 0 & d^{q_0+1}\end{array}\right),\] and \[W=\left(\begin{array}{cccc}0 & 0 & 0 & 1 \\0 & 0 & 1 & 0 \\0 & 1 & 0 & 0 \\1 & 0 & 0 & 0\end{array}\right).\]  Let $T=\{T_{a,c}:a,c\in\mathbb{F}_q\}$ and $M=\{M_d:d\in\mathbb{F}_q^{*}\}$.  Then $Sz(q)=\langle M, T, W\rangle$.  

The Suzuki curve $\mathcal{S}_q$ is the Deligne-Lusztig curve with automorphism group $Sz(q)$.  The curve $\mathcal{S}_q$ has a singlular model in $\mathbb{P}^2$ with affine equation  
\[y^q+y=x^{q_0}(x^q+x).\]
The genus of $\mathcal{S}_q$ is $q_0(q-1)$ and it has $q^2+1$ points over $\mathbb{F}_q$, making it optimal over the field.  Over $\mathbb{F}_{q^4}$, $\mathcal{S}_q$ is maximal, attaining the upper Weil bound \cite{Hansen}.  Smooth models for $\mathcal{S}_q$ in higher-dimensional spaces have been determined \cite{Ballico, Duursma}.  As in \cite{GKT}, a convenient model in $\mathbb{P}^4$ can be defined by the affine equations
\[y^q+y=x^{q_0}(x^q+x),\]
\[z = x^{2q_0+1}+y^{2q_0},\]and 
\[w=xy^{2q_0}+z^{2q+1}.\]  Let $[U:X:Y:Z:W]$ be a set of projective coordinates for $\mathbb{P}^4$, where for all $U\neq 0$ we have affine coordinates given by
\[x=\frac{X}{U}, \hspace{.25in} y=\frac{Y}{U}, \hspace{.25in} z=\frac{Z}{U}, \hspace{.25in} w=\frac{W}{U}.\]
For any $(x,y)\in\mathbb{F}_q^2$ satisfying $y^q+y=x^{q_0}(x^q+x)$, let $P_{(x,y)}$ denote the point $[1:x:y:z:w]\in\mathcal{S}_q(\mathbb{F}_q)$.  Let $P_{\infty}=[0:0:0:0:1]\in\mathcal{S}_q(\mathbb{F}_q)$.  

\begin{theorem}
There are locally recoverable codes $C(D,B)$ with availability $2$, and locality $(q-1,q-2)$ on the Suzuki curve $\mathcal{S}_q$

 \begin{enumerate}
 \item over $\mathbb{F}_q$, with length $n=q(q-1)$ and  dimension $k=(q-1)(q-2)$; and 
 \item over $\mathbb{F}_{q^4}$, with length $n=q(q-1)(q^2+2qq_0+q+1)$ and dimension $k=(q-1)(q-2)(q^2+2qq_0+q+1)$.  \end{enumerate}

\end{theorem}

\begin{proof}
Let $T_0=\{T_{0,c}:c\in\mathbb{F}_q\}\leq T$. It is straightforward to compute that $M_dT_{0,c}M_d^{-1}=T_{0,cd^{2q_0+1}}$, so $M$ normalizes $T_0$.  Thus, $T_0$ is a normal subgroup of $G=\langle M,T_0 \rangle$, and $G=T_0 \rtimes M$.  Thus, $G=(T_0 \rtimes M)\leq Sz(q)$.

Let $\mathcal{Y}_1=\mathcal{S}_q/T_0$.  The accompanying natural map $g_1:\mathcal{S}_q\rightarrow \mathcal{Y}_1$ is degree $q$ and fully ramified at  a single point $P_{\infty}\in\mathcal{S}_q(\mathbb{F}_q)$ \cite[Theorem 6.1]{GKT}.  Then, by \cite[Theorem 6.6]{GKT}, $\mathcal{Y}_1$ has genus $0$ and has an affine model as given in \cite{GKT}.

Let $\mathcal{Y}_2=\mathcal{S}_q/M$.  By \cite[Theorem 4.1]{GKT}, the natural map $g_2:\mathcal{S}_q\rightarrow \mathcal{Y}_2$ is degree $q-1$ and fully ramified below two points, $P_{(1,0)}$ and $P_{\infty}\in\mathcal{S}_2(\mathbb{F}_q)$.  Then $\mathcal{Y}_2$ has genus $q_0$ and has an affine model as given in \cite{GKT}.

Let $\mathcal{Y}=\mathcal{S}_q/G$.  Since $\mathcal{Y}_1$ covers $\mathcal{Y}$, it must be that $\mathcal{Y}$ has genus $0$ as well.  Let $g:\mathcal{S}_q\rightarrow \mathcal{Y}$ be the accompanying natural map.  We then have the diagram of curves as shown in Figure \ref{Suzuki_fig}.

\begin{figure}  
\begin{tikzpicture}[node distance = 2cm, auto]
  \node (Y) {$\mathcal{Y}$};
  \node (Y1) [node distance=1.4cm, left of=Y, above of=Y] {$\mathcal{Y}_1$};
   \node (Y2) [node distance=1.4cm, right of=Y, above of=Y] {$\mathcal{Y}_2$};
   \node (X) [node distance=1.4cm, left of=Y2, above of=Y2] {$\mathcal{S}_q$};
  \draw[->] (Y1) to node [swap] {$h_1$} (Y);
  \draw[->] (Y2) to node {$h_2$} (Y);
  \draw[->] (X) to node [swap] {$g_1$} (Y1);
  \draw[->] (X) to node {$g_2$} (Y2);
  \draw[->] (X) to node {$g$} (Y);
 \end{tikzpicture}
   \caption{Suzuki curve and its quotients used for constructing LRC(2) with balanced recovery sets}
\label{Suzuki_fig}
\end{figure}
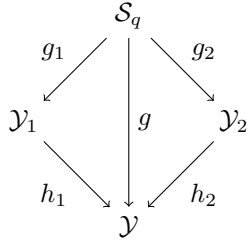

In Case (1), let $B=\mathcal{S}_q(\mathbb{F}_q)\setminus\{P_{\infty},g^{-1}\left(g(P_{(1,0)})\right)\}$.  Then $$n=|B|=|\mathcal{S}_q(\mathbb{F}_q)|-(q+1)=q(q-1).$$In the construction from Theorem \ref{construction}, we may choose $l=0$, so $\mathcal{L}(D)$ is simply the set of constant functions on $\mathcal{Y}$, giving $\ell(D)=1$ and yielding the stated dimension. 

In Case (2), let $B=\mathcal{S}_q(\mathbb{F}_{q^4})\setminus\{P_{\infty},g^{-1}\left(g(P_{(1,0)})\right)\}$.  Then $$n=|B|=|\mathcal{S}_q(\mathbb{F}_{q^4})|-(q+1)=q(q-1)(q^2+2qq_0+q+1).$$Since $\mathcal{Y}$ has genus $0$, we may let $l=q^2+2qq_0+q$, $D=lP_{\infty}$, and assume that functions in $\mathcal{L}(D)$ are represented by polynomials of degree less than or equal to $l$.  We then have $\ell(D)=q^2+2qq_0+q+1$, so $k=(q-1)(q-2)(q^2+2qq_0+q+1)$ by Theorem \ref{construction}.

\end{proof}

\begin{remark}
 In both Case (1) and Case (2), the bound on the minimum distance from Theorem \ref{construction} applies. However, without explicit generators for the given function field extensions, one cannot determine the degrees of the functions $x_i$ to give an explicit bound on $d$. Hence, some work around is required. Therefore we use the very modest bound from Proposition \ref{distance} to guarantee $d \geq 3$ in both cases.   
 
This illustrates a drawback of the `top-down' construction presented in Corollary \ref{aut_gp_construction} in that while the codes $C(D,B)$ might have reasonable classical parameters $n$, $k$, and $d$, one might not have access to the information needed to provide a good estimate of the minimum distance $d$. 
\end{remark}

\begin{remark} There is no benefit to considering the code defined over $\mathbb{F}_{q^2}$ because $|\mathcal{S}_m(\mathbb{F}_{q^2})|=|\mathcal{S}_m(\mathbb{F}_{q})|$.  The field $\mathbb{F}_{q^4}$ is a good choice for increasing length and dimension of the code because, as mentioned above, $\mathcal{S}_m$ is maximal over this field.
\end{remark}

In this section, we used Suzuki curves to determine LRC(2)s with recovery sets which are more balanced in size than those constructed in Section \ref{GK_section}. However, the quotient curve construction does not naturally provide explicit expressions for the bases of functions. In this case, equations and explicit realizations of their function fields are known for the quotient curves \cite{GKT}. However, the necessary functions $x_1$ and $x_2$ that generate the function fields of the quotient curves are unknown, even for this very well-studied curve.  This is a larger issue with the quotient curve construction--knowledge of the top curve and the existence, genera, and even models of the quotient curves does not give full information about the functions that generate the associated function field extensions, and their degrees as maps to $\mathbb{P}^1$.  Hence, even given a curve $\mathcal{X}$ with many points and a large automorphism group, it can be difficult to generate useful codes from the quotients, and, as mentioned in Section \ref{construction_subsection}, extending this construction to more than two subgroups of Aut$(\mathcal{X}$) faces additional obstacles.  This motivates us to seek good general examples for LRC($t$)s from fiber product constructions.  In the next section, we consider a general fiber product construction that is both explicit and gives rise to balanced recovery sets. Moreover, it naturally leads to even more recovery sets for each position.

\section{LRCs with multiple recovery sets from fiber products of Artin-Schreier curves} \label{LRCt_section}

In \cite{vanderGeer1996}, van der Geer and van der Vlugt develop several constructions of fiber products of Artin-Schreier curves with many points.  In particular, they construct maximal curves via the fiber product of several smaller genus maximal curves, in both characteristic 2 and odd characteristic.  These constructions are a very natural source of curves for locally recoverable codes with many recovery sets.

The simplest of these constructions is given in \cite[Section 3, Method I]{vanderGeer1996}. Let $p$ be prime, $2h$ an even natural number, and $q=p^{2h}$; \cite{vanderGeer1996} also gives a similar construction odd powers of $p$.  Let $A=\{a\in\mathbb{F}_q :a^{p^{h}}+a=0\}$.  As the kernel of the $\mathbb{F}_p$-linear trace map $\mathbb{F}_q/\mathbb{F}_{\sqrt{q}}$,  $A$ is an $h$-dimensional $\mathbb{F}_p$ vector space.  Let $\{a_1,a_2, \dots, a_{h}\}$ generate $A$ over $\mathbb{F}_p$.  Then the curves \[\mathcal{Y}_{a_i}: y^p-y=a_ix^{p^{h}+1}\] each have genus $\frac{(p-1)\sqrt{q}}{2}$ and have $pq+1$ points over $\mathbb{F}_q$, with one point $\infty_{\mathcal{Y}_{a_i}}$ at infinity.  

Let $t$ be an integer with $1\leq t \leq h$.  Then consider the natural map $h_i: \mathcal{Y}_{a_i}\rightarrow \mathbb{P}^1_x$ given by projection onto the $x$ coordinate, where $\infty_x$ represents the point at infinity on the projective line $\mathbb{P}^1_x$ and $\infty_{\mathcal{Y}_{a_i}}\mapsto \infty_x$.  These are all degree-$p$ Artin-Schreier covers of $\mathbb{P}^1_x$, fully ramified above $\infty_x$.  

Define $\mathcal{X}$ to be the fiber product of these curves $\mathcal{Y}_{a_i}$ over $\mathbb{P}^1_x$; i.e.,
\[\mathcal{X}=\mathcal{Y}_{a_1}\times_{\mathbb{P}_x^1} \mathcal{Y}_{a_2}\times_{\mathbb{P}_x^1} \dots \times_{\mathbb{P}_x^1} \mathcal{Y}_{a_t}.\] The corresponding maps $g_i: \mathcal{X}\rightarrow \mathcal{Y}_{a_i}$ are degree $p^{t-1}$, ramified only above $\infty_{\mathcal{Y}_{a_i}}$.  Let $\infty_{\mathcal{X}}$ be the single point above $\infty_{\mathcal{Y}_{a_i}}$ on $\mathcal{X}$.

\begin{figure} 
\begin{tikzpicture}[node distance = 2cm, auto] 
  \node (Y) {$\mathbb{P}_x^1$};
  \node (Y1) [node distance=1.4cm, left of=Y, above of=Y] {$\mathcal{Y}_{a_2}$};
    \node (Ys) [node distance=1.4cm, right of=Y, above of=Y] {$\mathcal{Y}_{a_t}$};
   \node (X) [node distance=1.4cm, left of=Ys, above of=Ys] {$\mathcal{X}$};
 \node (Y2) [node distance=1.4cm, left of=Y1, left of=Ys, below of=X] {$\mathcal{Y}_{a_1}$};
  \draw[->] (Y1) to node {$h_2$} (Y);
   \draw[->] (Y2) to node [swap] {$h_1$} (Y);
  \draw[->] (Ys) to node {$h_t$} (Y);
   \draw[->] (X) to node [swap] {$g_1$} (Y2);
  \draw[->] (X) to node {$g_2$} (Y1);
  \draw[->] (X) to node {$g_t$} (Ys);
  \path (Y1) -- node[auto=false]{\ \ \ \ \ \ \ldots } (Ys);
\end{tikzpicture}
  \caption{Curves for locally recoverable codes with availability $t$}
  \label{LRCt_fig}
  \end{figure}
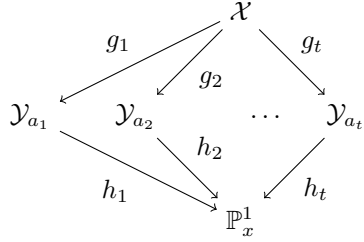

As shown in \cite[Theorem 3.1]{vanderGeer1996}, the curve $\mathcal{X}$ has genus $\frac{(p^t-1)\sqrt{q}}{2}$ and $|\mathcal{X}(\mathbb{F}_q)|=p^tq+1$, making $\mathcal{X}$ maximal over $\mathbb{F}_q$.  
Explicitly, we can write $$B=\{(x,y_1,y_2,\dots, y_t)\in\mathbb{F}_q^{t+1}: y_i^p+y_i=a_ix^{\sqrt{q}+1}\}.$$  The functions $g_i:\mathcal{X}\rightarrow \mathcal{Y}_{a_i}$ are given by $$g_i(x,y_1,y_2,\dots, y_t)=(x,y_i)$$ and the functions $\tilde{g}_i:\mathcal{X}\rightarrow \tilde{\mathcal{Y}}_{a_i}$ are given by $$\tilde{g}_i(x,y_1,y_2,\dots, y_t)=(x,y_1,y_2,\dots, y_{i-1},y_{i+1},\dots, y_t).$$  For each $i$, the map $g_i$ has degree $p^{t-1}$. The function $y_i$ has degree $d_{y_i}=p^h+1$, since for each $\alpha, \beta\in\mathbb{F}_q$ with $\beta\neq 0$ and $\alpha^p+\alpha = a_i\beta^{p^h+1}$, there are $p^h+1$ points $Q_j=(\zeta^k\beta,\alpha)\in\mathcal{Y}_{a_i}(\mathbb{F}_q)$, where $\zeta^{p^h+1}=1$ and $1\leq k\leq p^h+1$.

In applying the code construction, we use the divisor $D=l\infty_x$, meaning $\mathcal{L}(D)$ is the set of polynomials in $x$ over $\mathbb{F}_q$ of degree $\leq l$.  Then functions leading to codewords are 
\[V=\textrm{Span}\{x^jy_1^{e_1}y_2^{e_2}\dots y_t^{e_t}:0\leq j\leq l, 0\leq e_i\leq p-2\}.\]

Let $P_i=(\alpha,\beta_1,\beta_2,\dots, \beta_t)\in B$.  Then, returning to the notation of Section \ref{recovery_issues}, we have $B_{ij}$ is the set of positions corresponding to the points in $\{(\alpha,y_1,y_2,\dots, y_t)\in B:~ y_k=\beta_k~ \forall ~k\neq j\}.$  We then have $|B_{ij}|=p$. On points corresponding to the positions in $B_{ij}$, any function in $V$ varies as a polynomial in $y_j$ of degree at most $(p-2)$ and can therefore be interpolated by knowing its values on any $p-1$ points.

Choosing $l$ for maximum dimension with guaranteed positive minimum distance, we get the following theorem.

\begin{theorem} \label{AS Code} The code $C\left((q-tp^h+2tp^{h-1}-t)\infty_x,B\right)$ constructed from the fiber product $\mathcal{X}$ of the Artin-Schreier curves above is a locally recoverable $[n,k,d]$ code over $\F_{q}$ with availability $t$ and locality $\left(p-1,p-1,\dots,p-1 \right)$ where
\begin{align*}
n & = p^tq, \\
k &= (q-tp^h+2tp^{h-1}-t+1)(p-1)^t, and \\
d & \geq 2tp^{t-1}. \\
\end{align*} 
\end{theorem}
\begin{proof}
The construction in Theorem \ref{construction} gives rise to a code of length $n=p^tq$, where $B=\mathcal{X}(\mathbb{F}_q)\setminus \{\infty_{\mathcal{X}}\}$.  The set $S=g(B)$ has size $s=q$.  To achieve positive minimum distance, we then take $l=q-tp^h+2tp^{h-1}-t$, where $g_{\mathcal{X}}=0$ so $m=q-tp^h+2tp^{h-1}-t+1$, giving a code of dimension $k= (q-tp^h+2tp^{h-1}-t+1)(p-1)^t$.  Each coordinate has $t$ disjoint recovery sets of size $(p-1)$, given by the fibers $\tilde{g}_j^{-1}(P)$ for $P\in\tilde{\mathcal{Y}}_{a_j}$ where $1\leq j\leq t$.  The minimum distance is bounded by $d\geq n-p^t(q-tp^h+2tp^{h-1}-t)-t(p-2)p^{t-1}(p^h+1)=2tp^{t-1}$.
\end{proof}

\begin{remark}
Though it does not seem to have been remarked before, it is straightforward to see that when $t=h$, then the curve $\mathcal{X}$ has genus $\frac{(\sqrt{q}-1)\sqrt{q}}{2}$ and is maximal over $\mathbb{F}_{q}$.  Since the Hermitian curve $\mathcal{H}_{\sqrt{q}}$ is the only curve of this genus maximal over $\mathbb{F}_q$~\cite{RuckStichtenoth}, we have that $\mathcal{X}\cong\mathcal{H}_{\sqrt{q}}$.  Therefore the construction in Theorem \ref{AS Code} is yet another example of a code on the Hermitian curve with interesting properties.

\begin{example}
\label{example:fiber3} 

We construct a locally recoverable code with three recovery sets of size two for each position, letting $p=h=t=3$ in the above construction. Consider the field 
\[ \mathbb{F}_{3^6}\cong \mathbb{F}_3[x]/\langle x^6 + 2x^4 + x^2 + 2x+2\rangle. \]

As the construction indicates, we need roots of the polynomial $x^{27}+x$ that generate a $3$-dimensional vector space over $\mathbb{F}_{3}$. We choose the following roots: 
\begin{align*} 
2a^3 + a+1 & = a^{350}:=a_1 \\
a^4 + a^2 +1 & = a^{210}:=a_2 \\
2a^5 + a^3 + a^2 +1 &=a^{490}:=a_3. 
\end{align*} 
Now the components of the fiber product are the curves $\mathcal{Y}_{a_1}, \mathcal{Y}_{a_2}, \mathcal{Y}_{a_3}$, given below: 
\begin{align*} 
y^3+y &= a^{350}x^{28} \\
y^3+y &= a^{210}x^{28} \\
y^3+y &= a^{490}x^{28}. \\
\end{align*} 
The code has length $n=19683$, dimension $k=5600$, and minimum distance $d\geq 54$. 

\end{example} 
\end{remark}

\section{Conclusion} \label{Conclusion}

In this paper, we detailed a construction of locally recoverable codes from fiber products of algebraic curves, building on the work of Barg et. al. \cite{Barg16}. This construction results in different bounds (from those in \cite{Barg16}) for minimum distance in the case $t=2$ and allows for $t>2$ recovery sets.  This gives rise to several new families of locally recoverable codes, specifically those from generalized Giulietti and Korchm\'{a}ros (GK) curves, Suzuki curves, and the maximal curves of van der Geer and van der Vlugt, including an LRC($t$) from the Hermitian curve $\mathcal{H}_{p^t}$. The code construction from generalized GK curves is explicit, but the recovery sets have vastly different cardinalities (an advantage or disadvantage, depending on the perspective taken). In contrast, the codes from the Suzuki curves provide balanced recovery sets, but the codes themselves are not easily explicitly constructed. The construction using the curves of van der Geer and van der Vlugt provides both explicit code construction and balanced recovery sets, of which arbitrarily many are available.

{\it Acknowledgements} We would like to thank Bjorn Poonen and Rachel Pries for helpful comments and Paul Myrow for graphical assistance. In addition, we would like to acknowledge the hospitality of IPAM and the organizers of the 2016 Algebraic Geometry for Coding Theory and Cryptography Workshop: Everett Howe, Kristin Lauter, and Judy Walker.


\medskip
Received xxxx 20xx; revised xxxx 20xx.
\medskip

\end{document}